\documentclass[12pt]{article}
\usepackage{hyperref}
\usepackage{epsfig}
\usepackage{amsfonts}
\usepackage{amssymb}
\usepackage{amsmath}
\usepackage{amsthm}
\usepackage{latexsym}
\usepackage{graphicx}

\newtheorem{theorem}{Theorem}[section]
\newtheorem{lemma}[theorem]{Lemma}
\newtheorem{conjecture}[theorem]{Conjecture}
\newtheorem{corollary}[theorem]{Corollary}
\newtheorem{proposition}[theorem]{Proposition}
\newtheorem{problem}[theorem]{Problem}
\newtheorem{remark}[theorem]{Remark}

\DeclareMathOperator{\DV}{DV}
\DeclareMathOperator{\vol}{vol}
\DeclareMathOperator{\diam}{diam}

\newcommand{\MK}{\mathcal K}
\newcommand{\ML}{\mathcal L}
\newcommand{\MCR}{\mathcal R}
\newcommand{\E}{\mathbb R}
\newcommand{\R}{\mathbb R}
\newcommand{\Z}{\mathbb Z}
\renewcommand{\vec}[1]{\boldsymbol{#1}}
\newcommand{\rl}{r_{\!_L}}

\begin{document}

\title{The Isodiametric Problem with Lattice-Point
Constraints\footnote{The first author was supported in part by
Direcci\'on General de Investigaci\'on (MEC) MTM2004-04934-C04-02
and by Fundaci\'on S\'eneca (C.A.R.M.) 00625/PI/04. The second and
the third author were supported by the Deutsche
Forschungsgemeinschaft (DFG) under grant SCHU 1503/4-2. During the
work on this paper the third author was partially supported by the
Edmund Landau Center for Research in Mathematical Analysis and
Related Areas, sponsored by the Minerva Foundation (Germany), and he
was partially supported by the Netherlands Organization for
Scientific Research under grant NWO 639.032.203.}}

\author{M. A. Hern\'andez Cifre, A. Sch\"urmann, F.
Vallentin}

\date{}

\maketitle

\begin{abstract}
In this paper, the isodiametric problem for centrally symmetric convex
bodies in the Euclidean $d$-space $\E^d$ containing no interior non-zero
point of a lattice $L$ is studied. It is shown that the intersection of a
suitable ball with the Dirichlet-Voronoi cell of $2L$ is extremal, i.e.,
it has minimum diameter among all bodies with the same volume. It is
conjectured that these sets are the only extremal bodies, which is proved
for all three dimensional and several prominent lattices.
\end{abstract}

\smallskip



\bigskip

\section{Introduction}

Let $\E^d$ be the $d$-dimensional Euclidean space endowed with standard
norm $\|\cdot\|$ and inner product $\langle \cdot ,\cdot \rangle$. We
denote the set of full rank lattices by $\ML^d$, where a (full rank) {\em
lattice} $L$ of $\E^d$ is a set of the form $L = A\mathbb{Z}^d$ with $A
\in \mbox{GL}_d(\mathbb{R})$. The columns of $A$ are called a {\em basis}
of $L$. The determinant $\det L = |\det A|$ of the lattice is independent
of the chosen basis.

For a lattice $L$, let $\MK_L$ be the family of all centrally symmetric
convex bodies, that is, compact convex sets $K$ with $K = -K$, which do
not contain a non-zero lattice point in their interior. We denote the
volume ($d$-dimensional Lebesgue measure) of a convex body $K \in \MK_L$
by $\vol K$ and its diameter by $\diam K=2\max_{\vec{x}\in K}
\|\vec{x}\|$. Notice that for centrally symmetric sets the diameter is
twice the circumradius.

{\em Minkowski's first fundamental theorem} (see e.g. \cite[\S~5,
Th.~2]{gru-lek} and \cite{minkowski}) gives an upper bound for the
volume of a convex body $K \in \MK_L$:
\begin{theorem}[Minkowski, 1891] \label{mink-theorem}
If $L \in \ML^d$ and $K\in \MK_L$ then
\[
\vol K \leq 2^d\det L.
\]
\end{theorem}

Here we consider the following problem:

\begin{problem}[The isodiametric problem with lattice point constraints] \label{iso-problem}
Given a lattice $L\in\ML^d$ and a real number $V\in (0,2^d\det L]$,
determine the minimum diameter
\[
\diam_L(V)=\min\left\{\diam K : K \in \MK_L, \vol K = V\right\}
\]
and the bodies $K\in\MK_L$ for which this minimum is attained.
\end{problem}

Since the problem is trivial for $d = 1$, we consider it only for $d \geq
2$. In Theorem~\ref{satz1} we solve it partially by giving a description
of convex bodies attaining the minimum. We conjecture
(Conjecture~\ref{conj1}) that these convex bodies are the only ones
attaining the minimum and we show that the conjecture is valid in many
cases, e.g., for a wide class of lattices
(Corollary~\ref{coro_small-radii-theorem} and
Proposition~\ref{prop:ridge}) and for many values of $V$
(Corollary~\ref{small-radii-theorem}). In particular we give a complete
answer for all lattices of dimension $d\leq 3$, for the integral lattice
$\Z^d$, the Leech lattice $\Lambda_{24}$, all root lattices $\mathsf{A}_d$
with $d \geq 2$, $\mathsf{D}_d$ with $d \geq 3$, $\mathsf{E}_d$ with
$d=6,7,8$ and all their reciprocals (Theorem \ref{t:sol_lattices}). The
$2$-dimensional case was already completely solved in \cite{hern-scott} by
different methods.

\section{Main Results}\label{sec:results}

Before stating our main results we need some more notation. For a
lattice $L \in \ML^d$ we consider its {\it Dirichlet-Voronoi cell}
\[
\DV(L)=\bigl\{\vec{x}\in\E^d:\|\vec{x}\|\leq\|\vec{x}-\vec{y}\|\mbox{ for
all }\vec{y}\in L\bigr\}.
\]
Its volume is equal to $\det L$, its circumradius is equal to the {\it
inhomogeneous minimum} (also called covering radius) of $L$,
\[
\mu(L)=\min_{\vec{y} \in L} \max_{\vec{x} \in \E^d} \|\vec{x} -
\vec{y}\| =\min\bigl\{\mu\in\R:L + B_d(\mu) = \E^d\bigr\},
\]
and its inradius (also called packing radius) is half of the {\it
homogeneous minimum} of $L$,
\[
\lambda(L)=\min_{\vec{y} \in L\setminus \{\vec{0}\}} \|\vec{y}\| =
\min\bigl\{\lambda \in \R : B_d(\lambda) \cap L \neq\{\vec{0}\}\bigr\}.
\]
Here $B_d(r)=\bigl\{\vec{x}\in \E^d : \|\vec{x}\| \leq r\bigr\}$
denotes the $d$-dimensional ball centered at the origin with radius
$r$. For sets $A,B\subset\E^d$ we write
$A+B=\{\vec{a}+\vec{b}:\vec{a}\in A,\,\vec{b}\in B\}$ to denote the
Minkowski addition (vector sum). We write $\vec{y}+K$ instead of
$\{\vec{y}\}+K$.  For a set $A \subset \E^d$ and a real number
$\alpha$ we define $\alpha A=\{\alpha \vec{a}:\vec{a}\in A\}$. Then,
if $K$ is a $d$-dimensional convex body we have $\vol (\alpha K) =
\alpha^d \vol K$ and $\diam (\alpha K) = \alpha \diam K$.

Clearly, $\DV(L)$ is a {\it parallelohedron}, i.e., a convex polytope
which tiles $\E^d$ by (lattice) translations. Furthermore, every
$(d-1)$-dimensional face (facet) of $\DV(2L)$ contains exactly one
lattice point in its relative interior, and it is centrally symmetric
with respect to this point (see e.g. \cite[\S~12]{gru-lek}).

For $V \in (0,2^d\det L]$ we define the convex body $K_L(V)$ by
choosing the (unique) positive real number $\rl(V)$ such that
\begin{equation} \label{extremal-body}
K_L(V)=B_d\bigl(\rl(V)\bigr)\cap\DV(2L)
\end{equation}
has volume $V$ (see Figure \ref{extremal_sets} for an example of
$K_L(V)$ in $\E^3$).

\begin{theorem}\label{satz1}
Let $L \in \ML^d$ be a lattice and let $V \in (0,2^d\det L]$. Then for
all $K \in \MK_L$ with $\vol K=V$ the inequality $\diam K_L(V) \leq
\diam K$ holds.
\end{theorem}

We present a proof in Section \ref{sec:proofs}. We say that a convex body
$K\in\MK_L$ is an {\em extremal body} if it is a solution of Problem
\ref{iso-problem}, hence if $\diam K=\diam K_L(V)$.

\begin{conjecture} \label{conj1}
For $L \in \ML^d$ and $V \in (0,2^d\det L]$, $K_L(V)$ is the unique
extremal body.
\end{conjecture}

Notice that Conjecture~\ref{conj1} is trivially true if $\rl(V) \leq
\lambda(L)$ because of the classical isodiametric inequality without
lattice-point constraints (see e.g. \cite[p.~83]{BonFen}), which says
that for a fixed volume the ball is the only set with minimum
diameter. In Section \ref{sec:equality-cases} we verify Conjecture
\ref{conj1} for many particular cases:

\begin{theorem}\label{t:sol_lattices}
$K_L(V)$ is the only extremal body for any $V \in (0,2^d \det L]$ and for
the following lattices (for explicit descriptions we refer to
\cite{cs-1988}): all lattices in dimension $2$ and $3$, the integral
lattice $\Z^d$, the Leech lattice $\Lambda_{24}$, all root lattices
$\mathsf{A}_d$ with $d \geq 2$, $\mathsf{D}_d$ with $d \geq 3$,
$\mathsf{E}_d$ with $d = 6,7,8$ and all their reciprocals.
\end{theorem}

\section{Proof of Theorem~\ref{satz1}}\label{sec:proofs}

Our proof of Theorem~\ref{satz1} relies on an equivalent point of view.
Instead of minimizing the diameter among all convex bodies in $\MK_L$ with
fixed volume, we maximize the volume among all convex bodies in $\MK_L$
with fixed diameter. It turns out that extremal bodies also maximize
volume among all convex bodies in $\MK_L$ with fixed diameter.

For any $V \in (0, 2^d \det L]$ the minimum diameter of a body
$K\in\MK_L$ with volume $V$ is at most the diameter of $\DV(2L)$,
which is equal to $4\mu(L)$. Theorem~\ref{satz1} is equivalent to the
following statement.

\begin{theorem}\label{satz2}
Let $L\in\ML^d$ be a lattice and let $D$ be at most $4\mu(L)$.
Define $V$ by $2\rl(V)=D$. Then for every $K \in \MK_L$ with $\diam
K = D$ the inequality $\vol K\leq\vol K_L(V)$ holds.
\end{theorem}

Here we show that Theorem~\ref{satz2} implies Theorem~\ref{satz1}. Using
an analogous argument the other direction can be proved. For $V \in
(0,2^d\det L]$ let $K \in \MK_L$ be a convex body with $\vol K=V$. Suppose
that $\diam K_L(V)>\diam K$. Then Theorem~\ref{satz2} yields a
contradiction:
\[
V=\vol K_L(V)\geq\vol\left(\dfrac{\diam K_L(V)}{\diam
K}K\right)>\vol K=V.
\]
Hence, $\diam K_L(V) \leq \diam K$.

For the proof of Theorem~\ref{satz2} we recall some standard notions
from the Geometry of Numbers. A lattice $L\in\ML^d$ is called {\it
admissible} for a subset $K\subset\E^d$ if $K$ has no lattice point
except the origin in its interior. In particular the lattice $L$ is
admissible for all convex bodies in $\MK_L$. On the other hand, a
lattice $L\in\ML^d$ is called {\it packing lattice} for $K$ if ${\rm
int}(K+\vec{x})\cap{\rm int}(K+\vec{y})=\emptyset$ for all distinct
$\vec{x},\vec{y}\in L$. Then for a centrally symmetric convex body
$K$, a lattice $L$ is a packing lattice for $K$ if and only if it is
admissible for $2K$ (see e.g.~\cite[\S 20, Th.~1]{gru-lek}).

\begin{proof}[Proof of Theorem~\ref{satz2}]
As mentioned above, the case $\rl(V) \leq \lambda(L)$ is covered by
the classical isodiametric inequality without lattice-point
constraints. Thus we suppose $\rl(V)>\lambda(L)$.

Let $K \in \MK_L$ be a convex body with $\diam K=2\rl(V)$. By definition
of $\MK_L$ the lattice $L$ is admissible for $K$. Therefore $2L$ is a
packing lattice for~$K$. Since $\DV(2L)$ is a parallelohedron,
\[
\vol K=\vol\bigl((2L+K)\cap \DV(2L)\bigr)
\]
(roughly speaking, the Dirichlet-Voronoi cell $\DV(2L)$ contains the
full set $K$ ``in pieces"). On the other hand, since $\diam K=2\rl(V)$
we have $2\vec{y}+K\subseteq 2\vec{y}+B_d\bigl(\rl(V)\bigr)$ for all
$\vec{y}\in L$. Then,
\begin{equation}
\label{2lk} (2L+K)\cap\DV(2L)\subseteq
\bigl[2L+B_d\bigl(\rl(V)\bigr)\bigr] \cap \DV(2L).
\end{equation}
The right hand side of \eqref{2lk} is equal to
$B_d\bigl(\rl(V)\bigr)\cap\DV(2L)=K_L(V)$, because if $\vec{x}\in
\bigl[2\vec{y} + B_d\bigl(\rl(V)\bigr)\bigr]\cap \DV(2L)$ for
$\vec{y} \in L$, then by the definition of the Dirichlet-Voronoi
cell we have $\|\vec{x}\| \leq \|\vec{x} - 2\vec{y}\| \leq r_L(V)$.
Therefore, $\vol K_L(V)\geq\vol\bigl((2L+K)\cap \DV(2L)\bigr)=\vol
K$, as required.
\end{proof}

\section{Equality Cases}\label{sec:equality-cases}

In this section we further investigate equality cases. We give
geometric conditions on the vertices (Corollary
\ref{coro_small-radii-theorem}) and on the facets (Proposition
\ref{prop:ridge}) of the Dirichlet-Voronoi cell $\DV(2L)$ which
assure the validity of Conjecture~\ref{conj1} for the corresponding
lattice $L$. As a consequence of these results we obtain a proof of
Theorem \ref{t:sol_lattices}.

Let $L\in\ML^d$ be a lattice and let $V \in (0,2^d\det L]$. Then
$\rl(V)\leq 2\mu(L)$.  We define
\begin{equation}\label{eqn:MCR}
\MCR=\bigl\{\vec{x}\in\E^d:\|\vec{x}-\vec{y}\|>\rl(V),\mbox{ for all
} \vec{y} \in 2L\bigr\}.
\end{equation}

\begin{lemma}
Let $K\in\MK_L$ be an extremal body. Then
\begin{enumerate}
\item[i)] $\E^d =\bigl(2L + K_L(V)\bigr) \cup{\MCR}$,
\item[ii)] the intersection $(2L + K)\cap\MCR$ is empty,
\item[iii)] $\E^d=(2L+K)\cup{\MCR}$.
\end{enumerate}
\end{lemma}

\begin{proof}
For every $\vec{x} \in \E^d$ there exists $\vec{y}\in L$ such that
$\vec{x}\in2\vec{y}+\DV(2L)$. From the definition of $\DV(2L)$ we
know that $2\vec{y}$ is a nearest point of the lattice $2L$ to
$\vec{x}$. If $\vec{x}\not\in2\vec{y}+K_L(V)$, then
$\vec{x}\not\in2\vec{y}+B_d\bigl(\rl(V)\bigr)$. Consequently, since
$2\vec{y}$ is a nearest point to $\vec{x}$, we have $\vec{x}\not\in
2L+B_d\bigl(\rl(V)\bigr)$. Hence, $\vec{x}\in\MCR$. This shows i).

ii) is obvious since $\diam K=2\rl(V)$.

For iii) suppose that $(2L+K)\cup{\MCR}$ is strictly contained in
$\E^d$. Then there exists a set $B$ with positive volume such that
$2L+B$ does not intersect $(2L+K)\cup{\MCR}$. The latter is
invariant with respect to translations of $2L$ which contradicts the
extremality of $K$.
\end{proof}

Let $K \in \MK_L$ be an extremal body with respect to the lattice
$L$ and a value $V \in (0,2^d\det L]$. Since $2L$ is admissible for
$K$, there exists a closed halfspace $G_{\vec{y}}^-$ for every
$\vec{y} \in L \setminus \{\vec{0}\}$ with bounding hyperplane
$G_{\vec{y}}$ through $\vec{y}$ so that $K \subseteq G_{\vec{y}}^-$.
Since $K$ is contained in $B_d\bigl(\rl(V)\bigr)$, we have the
representation
\begin{equation}
\label{eqn:extremalK}
K=\bigcap_{\substack{\vec{y}\in L\setminus\{\vec{0}\},\\[1mm]
\|\vec{y}\| < \rl(V)}} G_{\vec{y}}^- \cap B_d\bigl(\rl(V)\bigr).
\end{equation}
Notice that we do not have to care about separation of lattice points
outside or on the boundary of $B_d\bigl(\rl(V)\bigr)$.

We now consider those $\vec{y}\in L \setminus \{\vec{0}\}$ with
$\|\vec{y}\|<\rl(V)$ which are facet-centers of the Dirichlet-Voronoi cell
$\DV(2L)$. Notice that all the facets of $\DV(2L)$ have a center since
they are centrally symmetric. In some cases we can prove that facet
defining hyperplanes of $\DV(2L)$ coincide with those of~$K$. In the
following $C_L\subset L$ denotes the set of lattice points being centers
of facets of $\DV(2L)$. Let the facets of $\DV(2L)$ be defined by
hyperplanes $H_{\vec{y}}$ passing through $\vec{y}\in C_L$ and let
$H_{\vec{y}}^-$ denote the corresponding closed halfspaces bounded by
$H_{\vec{y}}$ and containing $\DV(2L)$, so that
\begin{equation}
\label{eqn:dv2l}
\DV(2L) = \bigcap_{\vec{y} \in C_L} H_{\vec{y}}^-
\end{equation}
is a non-redundant description of $\DV(2L)$.

The following proposition allows to prove Conjecture \ref{conj1} for
many values of $V$ and many lattices.

\begin{proposition}\label{prop:equalfacets}
Let $L\in\ML^d$ be a lattice and let $D$ be at most $4\mu(L)$.
Define $V$ by $2\rl(V)=D$. Let $K\in\MK_L$, given as in
\eqref{eqn:extremalK}, be extremal with respect to $L$ and $V$, and
let $\vec{y}\in C_L$ with $\|\vec{y}\| < \rl(V)$. If the facet
$F_{\vec{y}} = H_{\vec{y}}\cap\DV(2L)$ of $\DV(2L)$ intersects
$\E^d\setminus B_d\bigl(\rl(V)\bigr)$ then
$G_{\vec{y}}=H_{\vec{y}}$.
\end{proposition}

\begin{proof}
Suppose $G_{\vec{y}}\neq H_{\vec{y}}$ for some $\vec{y}\in C_L$
satisfying the assumptions of the proposition.
We construct a ball $B$ having the following properties:
\begin{enumerate}
\item[i)] $B\subset K_L(V)$,
\item[ii)] $B\cap K=\emptyset$,
\item[iii)] $B\cap\bigl[2L \setminus \{\vec{0}\}+ B_d\bigl(\rl(V)\bigr)\bigr]=\emptyset$.
\end{enumerate}
Then the lattice $2L$ is a packing lattice for $K \cup B$. Using the
argument which was applied in the proof of Theorem~\ref{satz2} we
obtain the strict inequality $\vol K<\vol(K\cup B)<\vol K_L(V)$,
which contradicts the assumption that $K$ is extremal.

It remains to construct $B$. Since the facet
$F_{\vec{y}}$ of $\DV(2L)$ intersects $\E^d\setminus
B_d\bigl(\rl(V)\bigr)$ by our hypothesis, there exists a vertex
$\vec{x}$ of $F_{\vec{y}}$ with $\|\vec{x}\|>\rl(V)$ (see
Figure \ref{fig1}).

\begin{figure}[ht]
\begin{center}
\includegraphics[width=6.5cm]{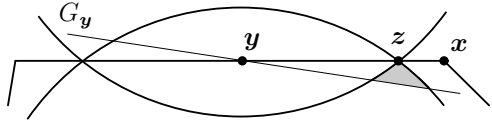}
\caption{The ball $B$ lies inside the shaded region.}\label{fig1}
\end{center}
\end{figure}

Since $F_{\vec{y}}$ is centrally symmetric with respect to
$\vec{y}$, either $\vec{x}$ or $2\vec{y}-\vec{x}$ lies in the open
halfspace $G_{\vec{y}}^+ = \E^d \setminus G_{\vec{y}}^-$. Without
loss of generality we assume $\vec{x} \in G_{\vec{y}}^+$.

On the line segment connecting $\vec{x}$ and $\vec{y}$ there is a
point $\vec{z}$ with $\|\vec{z}\| = \rl(V)$ (see Figure \ref{fig1}).
We have $\|\vec{z}-2\vec{y}\|=\rl(V)$ and,
since $\vec{z}$ lies in the relative interior of a facet of a
Dirichlet-Voronoi cell, for any other $\vec{y}'\in L \setminus\{\vec{0},\vec{y}\}$ it
holds $\|\vec{z}-2\vec{y}'\| > \rl(V)$. Let
$\vec{c}=\vec{z}-\varepsilon \vec{y}$ where $\varepsilon>0$ is
chosen so that $\vec{c}$ lies in the interior of
$\DV(2L)\cap G_{\vec{y}}^+ \cap B_d\bigl(\rl(V)\bigr)$. Thus there exists
$\delta>0$ sufficiently small such that the ball $B$ centered in
$\vec{c}$ with radius $\delta$ satisfies properties i)--iii).
\end{proof}

The following result is an immediate consequence of Proposition
\ref{prop:equalfacets}.
\begin{corollary}\label{small-radii-theorem}
Let $L\in\ML^d$ be a lattice and let $D$ be at most $4\mu(L)$. Define
$V$ by $2\rl(V)=D$.
If every facet of $\DV(2L)$ contains a vertex $\vec{x}$
with $\|\vec{x}\|>\rl(V)$, then $K_L(V)$ is the unique extremal body
with respect to $L$ and $V$.
\end{corollary}

So Corollary \ref{small-radii-theorem} proves Conjecture \ref{conj1}
for many values of $V$. We can apply Corollary \ref{small-radii-theorem}
also to prove Conjecture \ref{conj1} for certain classes of lattices and
any value of $V \in (0,2^d\det L]$:

\begin{corollary}\label{coro_small-radii-theorem}
Let $L \in \ML^d$ be a lattice such that every facet of $\DV(2L)$
contains a vertex $\vec{x}$ with $\|\vec{x}\|=2\mu(L)$. Then
$K_L(V)$ is the unique extremal body with respect to $L$ and any $V \in (0,2^d\det L]$.
\end{corollary}

\begin{proof}
The case $D<4\mu(L)$ is covered by
Corollary~\ref{small-radii-theorem}. So we assume that $D=4\mu(L)$.
If $K$ is an extremal body for the diameter $D=4\mu(L)$ then
$\E^d=2L+K$. Hence for arbitrarily small $\varepsilon>0$ the convex
body $K_{\varepsilon}=K\cap B_d(D-\varepsilon)$ is extremal for
diameter $D-\varepsilon$, since $\E^d=(2L+K_{\varepsilon})\cup\MCR$.
Thus applying Proposition~\ref{prop:equalfacets} to every facet of
$\DV(2L)\cap B_d(D-\varepsilon)$ we get that
$K_{\varepsilon}=\DV(2L)\cap B_d(D-\varepsilon)$. Hence $K=\DV(2L)$.
\end{proof}

Notice that even in dimension $3$ there are lattices to which
Corollary~\ref{coro_small-radii-theorem} cannot be applied. We give
an explicit example using the notation of {\em Selling parameters} from \cite{cs-1992}:
\begin{remark}
Let $L\in\ML^3$ be the lattice defined by the Selling parameters
$p_{01}=2$ and $p_{ij}=1$ for $i,j=0,\dots,3$ with $i\neq j$ and the
pair $(i,j)\neq(0,1)$. A Gram matrix of a basis $A$ of $L$ is for
example
\[
A^{\top}A=\left(\begin{array}{rrr}
4 & -2 & -1 \\
-2 & 4 & -1 \\
-1 & -1 & 3 \\
\end{array}\right).
\]
The Dirichlet-Voronoi cell of $L$ is a truncated octahedron (or
permutohedron, see Figure \ref{f:permutohedron}). All the vertices of its
$4$-gonal facet given by $1023$, $0123$, $0132$, $1032$ have norm
$\sqrt{35/24}$ whereas the vertex denoted by $0213$ has norm
$\sqrt{3/2}=\mu(L)>\sqrt{35/24}$.
\end{remark}

\begin{figure}[h]
\begin{center}
\includegraphics[width=5.3cm]{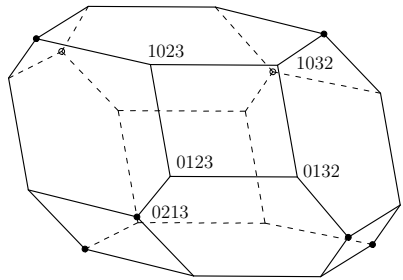}
\caption{The truncated octahedron.}\label{f:permutohedron}
\end{center}
\end{figure}

In order to overcome this problem in $\E^3$ and to solve more equality
cases, we prove the following proposition. It uses the fact that the
Dirichlet-Voronoi cell is a parallelohedron and hence every projection
along a $(d-2)$-face ({\it ridge}) is a centrally symmetric hexagon or, as
a limiting case, a parallelogram. This was independently proved by
McMullen \cite{mcmullen-1980,mcmullen-1981} and Venkov \cite{venkov-1954}.
The facets of the parallelohedron adjacent to the $4$ or $6$ translates of
a ridge are said to form a {\it $4$-belt} or {\it $6$-belt} respectively.

\begin{proposition}\label{prop:ridge}
Let $L\in\ML^d$ be a lattice and let $D$ be at most $4\mu(L)$. Define $V$
by $2\rl(V)=D$. Let $K\in\MK_L$, given as in \eqref{eqn:extremalK}, be
extremal with respect to $L$ and $V$. Let $\DV(2L)$ be given as in
\eqref{eqn:dv2l} and let $\vec{y}\in C_L$ be the center of a facet
belonging to a $6$-belt of $\DV(2L)$. If $G_{\vec{z}}=H_{\vec{z}}$ for all
$\vec{z}\in C_L\setminus\{\vec{y}\}$, then also $G_{\vec{y}}=H_{\vec{y}}$
and hence $K=K_L(V)$.
\end{proposition}

\begin{proof}
We suppose that $G_{\vec{y}}\neq H_{\vec{y}}$, respectively $K \neq
K_L(V)$. Consider the sets $A=G_{\vec{y}}^+ \cap K_L(V)$ and
$B=G_{\vec{y}}^- \cap \bigl(2\vec{y}+ K_L(V)\bigr)$, see Figure
\ref{f:ridge}. Since $K_L(V)$ is centrally symmetric, the isometry
$\vec{x}\mapsto 2\vec{y}-\vec{x}$ maps the closure $\overline{A}$ to
$\overline{B}$ and hence $\vol A = \vol B$. Clearly, $\vol(A\cap
K)=0$.

\begin{figure}[h]
\begin{center}
\includegraphics[width=4.8cm]{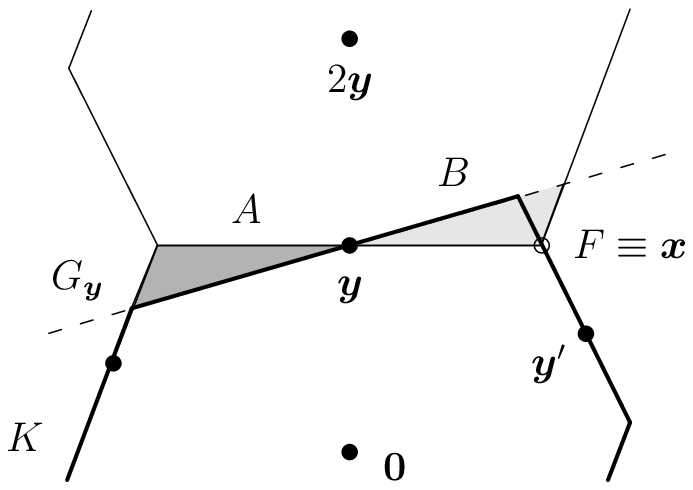}
\caption{}\label{f:ridge}
\end{center}
\end{figure}

By assumption, the facet $F_{\vec{y}}=H_{\vec{y}}\cap \DV(2L)$ of
$\DV(2L)$ with center~$\vec{y}$ contains two ridges $F$ and $2\vec{y}-F$
which are also ridges of other facets in a $6$-belt. Since $F_{\vec{y}}$
is centrally symmetric with respect to $\vec{y}$ we find a relative
interior point $\vec{x}$ of either $F$ or $2\vec{y}-F$ which is contained
in $G_{\vec{y}}^-\setminus G_{\vec{y}}$. Other\-wise $G_{\vec{y}}$ would
be equal to the affine hull of $\{\vec{y}\}\cup F$ and hence equal
to~$H_{\vec{y}}$.

Without loss of generality we assume $\vec{x}\in F$ and
$F=H_{\vec{y}'}\cap F_{\vec{y}}$. We can suppose that the points of
$F$ all lie in $B_d\bigl(r_L(V)\bigr)$ and
$2\vec{y}+B_d\bigl(r_L(V)\bigr)$, because otherwise we could apply
Corollary~\ref{small-radii-theorem} to show
$G_{\vec{y}}=H_{\vec{y}}$. Therefore there exists $\varepsilon > 0$
such that $\vec{x}+\varepsilon \vec{y}$ lies in the interior of
$G_{\vec{y}}^-\cap H_{\vec{y}'}^+$ and $2\vec{y}+K_L(V)$. So there
exists a ball centered in $\vec{x}+\varepsilon \vec{y}$ and
contained in $B\setminus K$ with positive volume. It shows that
$\vol(B\cap K)<\vol B$. Hence $K_L(V)$ has larger volume than $K$
which contradicts the extremality of $K$.
\end{proof}

Corollary \ref{coro_small-radii-theorem} and Proposition
\ref{prop:ridge} can be applied to many lattices so that the
isodiametric problem is solved for them completely, that is, for
every volume respectively every diameter. In the proof of Theorem
\ref{t:sol_lattices}  below we work out this argument for several
prominent lattices, where we did not try to be exhaustive.

\begin{proof}[Proof of Theorem \ref{t:sol_lattices}]
If the automorphism group of a lattice $L$ acts transitively on the
facet centers of $\DV(L)$, then every facet of $\DV(2L)$ contains a
vertex $\vec{x}$ with $\|\vec{x}\|=2\mu(L)$. Lattices with this
transitivity property are for example the integral lattice $\Z^d$
and the root lattices $\mathsf{A}_d$, $\mathsf{D}_d$ and
$\mathsf{E}_d$ (see \cite[Chap.~4, Chap.~22, Cor. to
Th.~5]{cs-1988}).

With a similar argument it can be shown that Corollary
\ref{coro_small-radii-theorem} applies to the Leech lattice
$\Lambda_{24}$ and the lattices $\mathsf{D}_d^*$. The facets of
$\DV(\Lambda_{24})$ are given by the lattice vectors of squared
length $4$ and $6$. The automorphism group of $\Lambda_{24}$ acts
transitively on each of these two sets. A vertex $\vec{x}$ of ``type
$\mathsf{A}_{24}$'' (see \cite[Chap.~23]{cs-1988}) satisfies
$\|\vec{x}\|=\mu(\Lambda_{24})$. This vertex is incident to $275$
facets which correspond to vectors of length $4$ and to $25$ facets
which correspond to vectors of length $6$. Hence in every facet a
vertex of type $\mathsf{A}_{24}$ can be found. The case
$\mathsf{D}_d^*$ is analogous.

Other lattices where Corollary \ref{coro_small-radii-theorem} can be
applied are all the $2$-dimen\-sional lattices and the reciprocals
$\mathsf{A}^*_d$, $\mathsf{E}_6^*$, $\mathsf{E}_7^*$ of the root lattices,
since they are examples of lattices $L$ for which every vertex $\vec{x}$
of $\DV(2L)$ satisfies $\|\vec{x}\|=2\mu(L)$ (see \cite[Chap.~4, Chap.~22,
Th.~7]{cs-1988}).

Finally in order to get the solution for $d=3$, we can use the knowledge
of all combinatorial types of Dirichlet-Voronoi cells in dimension~$3$. It
is well known (Fedorov, \cite{Fed85}) that there are five combinatorial
types of Dirichlet-Voronoi cells: {\em cube}, {\em hexarhombic
dodecahedron}, {\em rhombic dodecahedron}, {\em hexagonal prism} and the
{\em truncated octahedron}. The first four are degenerations of the last
one.

The vertices of $\DV(2L)$ can be partitioned into equivalence classes by
identifying those ones which are lattice translates or point reflections
of each other. Vertices of the same class have all the same distance to
the origin. For a vertex $\vec{x}$ of a facet with center $\vec{y}$, the
opposite vertex $2\vec{y}-\vec{x}$ of that facet belongs to the same
class. Using this symmetry it is easily checked for all but the truncated
octahedron, that every facet of $\DV(2L)$ contains a vertex $\vec{x}$ with
$\|\vec{x}\|=2\mu(L)$. Then Corollary \ref{coro_small-radii-theorem} can
be applied to derive that $K_L(V)$ is the unique extremal set.

In the remaining type, i.e., the one corresponding to the truncated
octahedron (see Figure~\ref{f:permutohedron}), we can apply
Proposition~\ref{prop:ridge}. We find that $\DV(2L)$ has three equivalence
classes of eight vertices each. Each hexagonal facet of the truncated
octahedron contains two opposite vertices of every class and each vertex
class has vertices in four quadrilateral facets (see points on vertices in
Figure \ref{f:permutohedron}). Assuming that only one class attains the
radius $2\mu(L)$, we know by Corollary~\ref{small-radii-theorem} that the
condition of Proposition~\ref{prop:ridge} holds, and therefore we obtain
$K_L(V)$ is the unique extremal set.
\end{proof}

\section{Some Consequences and Remarks}

For a lattice $L\in\ML^d$, the explicit isodiametric inequality for
convex bodies $K\in \MK_L$ can be stated by computing the volume
$V=\vol K_L(V)$ in terms of the diameter $\diam K_L(V)$. Notice that
the function defined as $f(r)=\vol\bigl(B_d(r)\cap \DV(2L)\bigr)$ is
clearly an increasing function of $r$. Then it follows that for any
convex body $K\in\MK_L$
\begin{equation} \label{explicit-inequality}
\vol K\leq f\left(\dfrac{\diam K}{2}\right),
\end{equation}
and equality holds when (and, in many cases, only when) $K=K_L(V)$.

For instance, in the case of the 3-dimensional Euclidean space and the
integral lattice $\mathbb{Z}^3$, the isodiametric inequality is expressed
in the following way (we write $D:=\diam K$ for the sake of brevity):
\[
\def\arraystretch{2}
\begin{array}{ll}
\text{If } 0<D\leq 2, & \vol K\leq \dfrac{\pi}{6}D^3\\
\text{If } 2 \leq D\leq 2\sqrt{2}, & \vol K\leq 2\pi \left( -\dfrac{D^3}{6}+\dfrac{3D^2}{4}-1 \right)\\
\text{If } 2\sqrt{2}\leq D\leq 2\sqrt{3}, & \vol K\leq 4\sqrt{D^2\!-8}+(3D^2-4)\arctan \dfrac{12-D^2}{4\sqrt{D^2-8}}\\
  & \hspace*{1.4cm} -\dfrac{2}{3}D^3\arctan \dfrac{D(12-D^2)}{(D^2+4)\sqrt{D^2-8}}
\end{array}
\]

The extremal sets for these inequalities, i.e., the sets with
maximum volume for different values of the diameter, are shown in
Figure~\ref{extremal_sets}.

\begin{figure}[h]
\begin{center}
\begin{tabular}{cc}
\includegraphics[width=4cm]{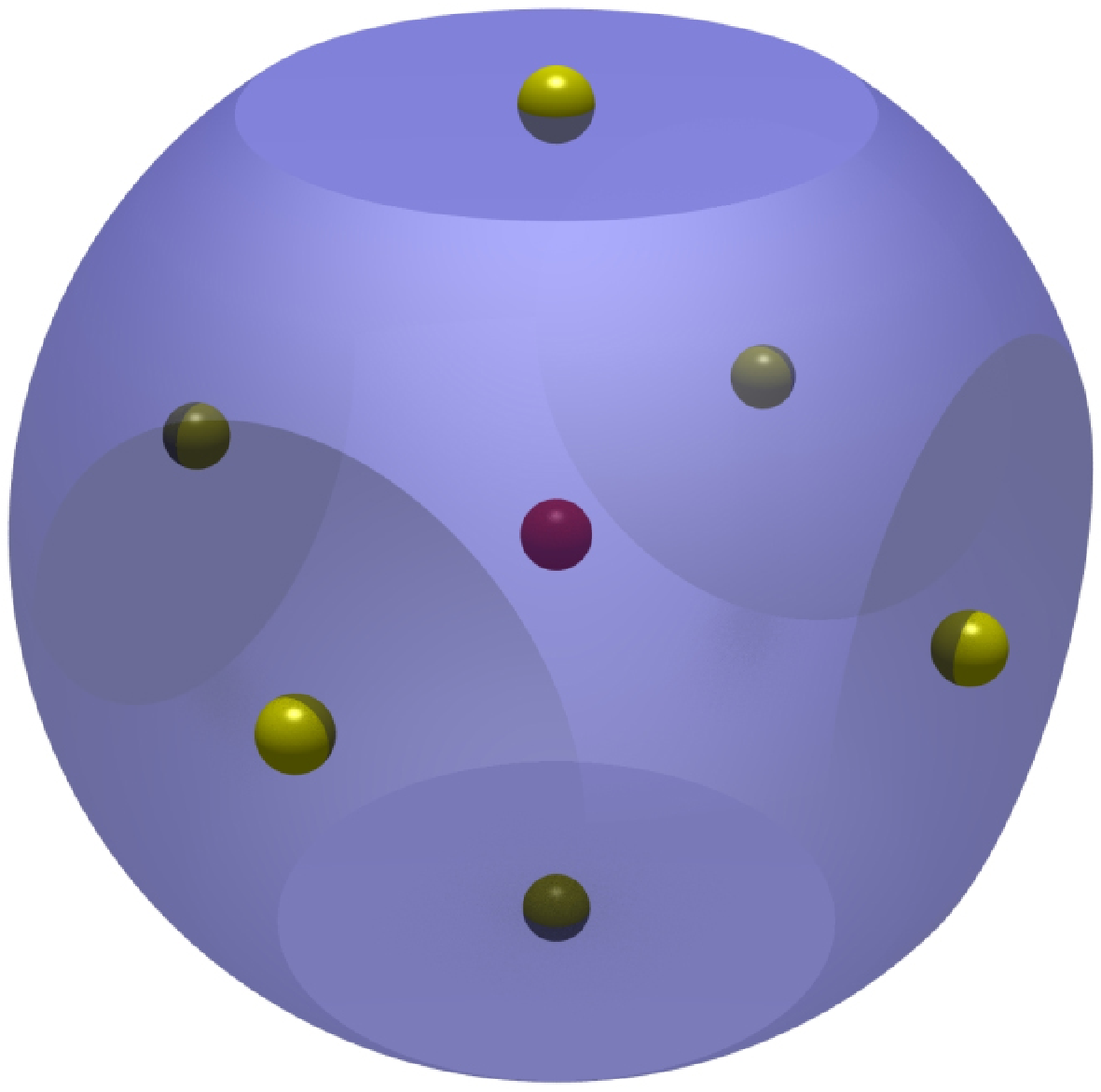}\hspace{0.7cm} &
\hspace*{0.7cm}\includegraphics[width=4cm]{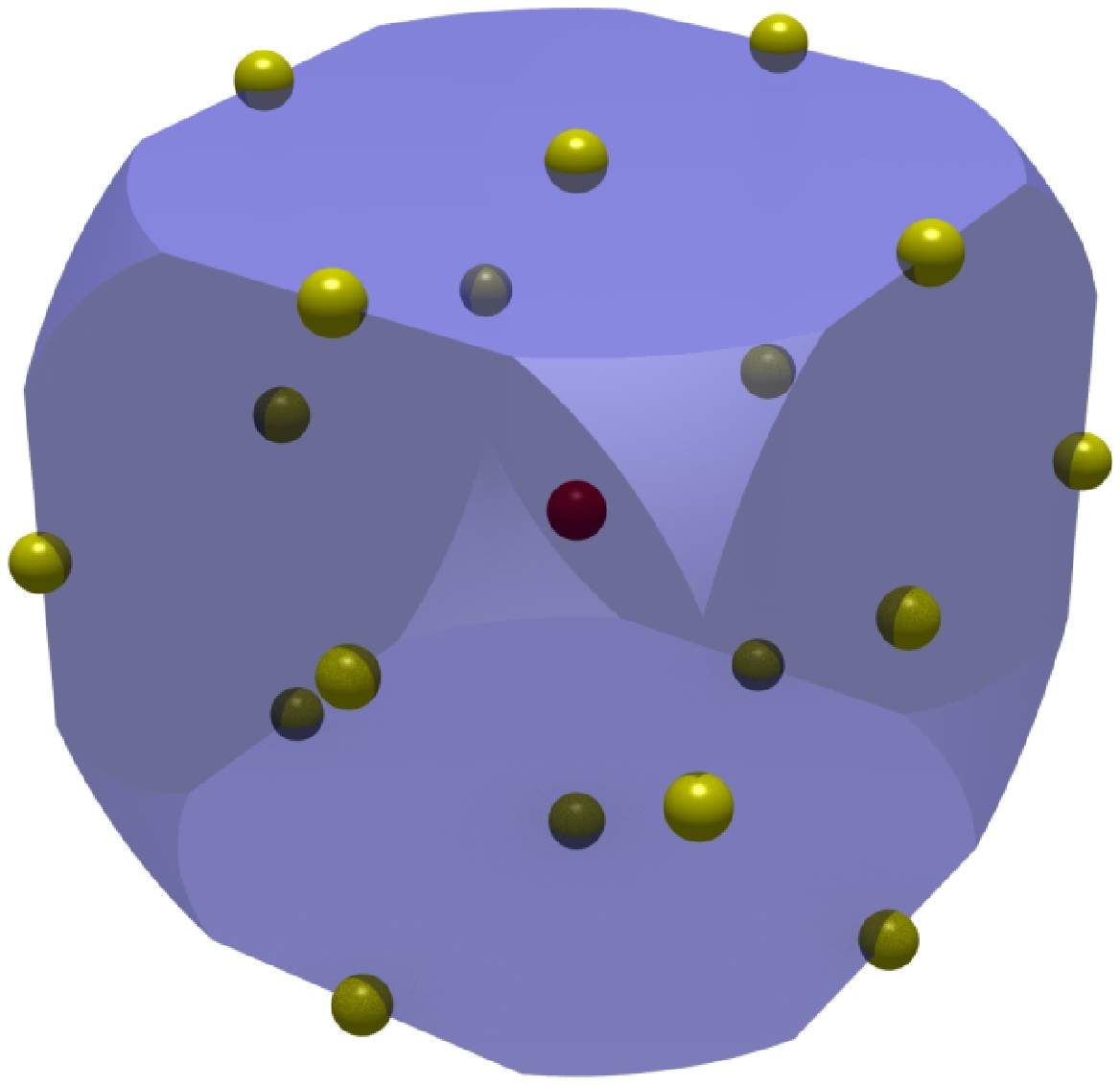}
\end{tabular}
\caption{Extremal sets in $\E^3$ for the integral lattice: when
$\diam K\leq 2\sqrt{2}$ (left) and $\diam K\geq 2\sqrt{2}$
(right).}\label{extremal_sets}
\end{center}
\end{figure}

Notice that, since the diameter is twice the circumradius for centrally
symmetric~$K$, all of the inequalities above also relate volume and
circumradius of~$K$.

\bigskip

{\footnotesize \noindent M. A. Hern\'andez Cifre, Departamento de
Matem\'aticas, Universidad de Murcia, Campus de Espinar\-do,
30100-Murcia, Spain, e-mail: mhcifre@um.es;

\noindent A. Sch\"urmann, Institut f\"ur Algebra und Geometrie,
Otto-von-Guericke Universit\"at Mag\-deburg, D39106 Magdeburg,
Germany, e-mail: achill.schuermann@ovgu.de;

\noindent F. Vallentin, Centrum voor Wiskunde en Informatica (CWI),
Kruislaan 413, 1098 SJ Amsterdam, The Netherlands,
e-mail:f.vallentin@cwi.nl

}

\end{document}